\documentclass{amsart}
\usepackage{graphicx}
\usepackage{amsmath}
\usepackage{amsthm}
\usepackage{amssymb,bbm}%
\usepackage[numbers, square]{natbib}

\newcommand{\E}{\mathbb E}

\newcommand{\R}{\mathbb{R}}
\newcommand{\N}{\mathbb{N}}

\renewcommand{\P}{\mathbb{P}}
\renewcommand{\S}{\mathbb{S}}

\newcommand{\range}{\mathop{\mathrm{Range}}}

\newcommand{\Var}{\mathop{\mathrm{Var}}\nolimits}

\newcommand{\Cov}{\mathop{\mathrm{Cov}}\nolimits}

\newcommand{\conv}{\mathop{\mathrm{conv}}\nolimits}

\newcommand{\eee}{{\rm e}}
\newcommand{\dd}{{\rm d}}

\newcommand{\eps}{\varepsilon}
\newcommand{\eqdistr}{\stackrel{d}{=}}

\newcommand{\todistr}{\overset{d}{\underset{n\to\infty}\longrightarrow}}

\newcommand{\ton}{\overset{}{\underset{n\to\infty}\longrightarrow}}

\newcommand{\bx}{\mathbf{x}}

\theoremstyle{plain}
\newtheorem{theorem}{Theorem}[section]

\newtheorem{corollary}[theorem]{Corollary}
\newtheorem{proposition}[theorem]{Proposition}

\theoremstyle{definition}

\newtheorem{example}[theorem]{Example}

\theoremstyle{remark}
\newtheorem{remark}[theorem]{Remark}

\begin{document}

\author{Zakhar Kabluchko}
\address{Zakhar Kabluchko, Institut f\"ur Mathematische Statistik,
Universit\"at M\"unster,
Orl\'eans--Ring 10,
48149 M\"unster, Germany}
\email{zakhar.kabluchko@uni-muenster.de}

\author{Alexander E.\ Litvak}
\address{Alexander E.\ Litvak, Department of Mathematical and Statistical Sciences, University of Alberta,  Edmonton, AB, T6G 2G1}

\email{aelitvak@gmail.com}

\author{Dmitry Zaporozhets}
\address{Dmitry Zaporozhets\\
St.\ Petersburg Department of
Steklov Institute of Mathematics,
Fontanka~27,
 191011 St.\ Petersburg,
Russia}
\email{zap1979@gmail.com}

\title[Mean width of regular polytopes]{Mean width of regular polytopes and expected maxima of correlated Gaussian variables}
\thanks{
The work of the third author is supported by the grant RFBR
13-01-00256 and by the Program of Fundamental
Researches of Russian Academy of Sciences ``Modern Problems of
Fundamental Mathematics''.}

\begin{abstract}
An old conjecture states that among all simplices inscribed in the unit sphere the regular one has the maximal mean width. An equivalent formulation is that for any centered Gaussian vector $(\xi_1,\dots,\xi_n)$ satisfying
$\E\xi_1^2= \dots =\E\xi_n^2=1$ one has
$$
 \E\,\max\{\xi_1,\dots,\xi_n\}\leq\sqrt{\frac{n}{n-1}}\, \E\,\max\{\eta_1,\dots,\eta_n\},
$$
where $\eta_1,\eta_2,\dots,$ are independent standard Gaussian variables.
Using this probabilistic interpretation we derive an \emph{asymptotic}
version of the conjecture. We also show that the mean width of the regular
simplex with $2n$ vertices is remarkably close to the mean width of the
regular crosspolytope with the same number of vertices. Interpreted probabilistically, our
result states that
$$
  1\leq\frac{\E\,\max\{|\eta_1|,\dots,|\eta_n|\}}{\E\,\max\{\eta_1,\dots,\eta_{2n}\}}
  \leq\min\left\{\sqrt{\frac{2n}{2n-1}}, \, 1+\frac{C}{n\, \log n} \right\},
$$
where $C>0$ is an absolute constant.
We also compute the higher moments of the projection length $W$ of the regular cube, simplex and crosspolytope onto a line with random direction, thus proving several formulas conjectured by S.~Finch. Finally, we prove distributional limit theorems for the length of random projection as the dimension goes to $\infty$. In the case of the $n$-dimensional unit cube $Q_n$, we prove that
$$
W_{Q_n} -  \sqrt{\frac{2n}{\pi}} \todistr {\mathcal{N}} \left(0, \frac{\pi-3}{\pi}\right),
$$
whereas for the simplex and the crosspolytope the limiting distributions are related to the Gumbel double exponential law.
\end{abstract}

\maketitle

\section{Conjecture on the mean width}
\subsection{Introduction}
The mean width of a compact convex body $K\subset \R^n$ is the expected length of a projection of this body onto a line with uniformly chosen, random direction. That is, the mean width equals $\E\, [W_K]$, where
$$
W_K = \sup_{t\in K} \langle U,t \rangle - \inf_{t\in K} \langle U,t \rangle,
$$
and $U$ is uniformly distributed on the unit sphere $\S^{n-1}\subset \R^n$.

How should $n+1$ points be arranged on the $(n-1)$-dimensional unit sphere so as to maximize the mean width of their convex hull? An old conjecture states (see~\cite[Section~9.10.2]{GK94}) that the arrangement must be \emph{regular}.

The mean width is just a multiple of the first
intrinsic volume $V_1$, namely
\begin{equation}\label{eq:mean_width_V1}
V_1(K)= \sqrt \pi\, \frac{\Gamma(\frac{n+1}{2})}{\Gamma(\frac{n}{2})}\E\, [W_K];
\end{equation}
see~\cite[p.~210]{Sch-book}.
The first intrinsic volume has the advantage of not depending on the dimension of the surrounding space.
Hence the conjecture can be formulated as follows:
\begin{equation}\label{2011}
 \sup_{x_1,\dots,x_{n+1}\in\S^{n-1}}V_1(\conv(x_1,\dots,x_{n+1}))=V_1(T_n),
\end{equation}
where $T_n$ is a regular simplex with $n+1$ vertices inscribed in the sphere $\S^{n-1}$, and $\conv$ denotes the convex hull.

This question is surprisingly hard.  Several authors~\cite{eG52, aB63, aB65, cW68} assumed the existence of a proof, but the problem is still open. Besides very natural formulation in Convex Geometry this problem is very important in Information Theory,
as it is closely related to the the long-standing simplex code conjecture \cite{code}.

\subsection{Probabilistic statement}
The conjecture can be reformulated in terms of Gaussian processes in the following way.
Throughout the paper, $\eta=(\eta_1,\dots,\eta_n)$ denotes a standard Gaussian vector
in $\R^n$. Consider a compact set $K\subset\R^n$. Using the fact that the norm and the direction of $\eta$ are independent, it is not difficult to derive Sudakov's formula
\begin{equation}\label{1802}
 V_1(\conv(K))=\sqrt{2\pi}\, \E\,\sup_{x\in K}\langle \eta,x\rangle
\end{equation}
(see~\cite{vS76} for details and for a generalization to the infinite-dimensional case, or Theorem~\ref{theo:sudakov_general} in the present paper for a more general result). This probabilistic interpretation of the first intrinsic volume allows to reformulate the conjecture as follows.
\begin{proposition}
For every integer $n\geq2$ the following two statements are equivalent:
\begin{itemize}
\item[(i)] One has
\begin{equation}\label{2004}
\sup_{x_1,\dots,x_{n}\in\S^{n-2}}V_1(\conv(x_1,\dots,x_{n}))=V_1(T_{n-1}),
\end{equation}
and the equality is attained iff $x_1,\ldots,x_n$ are vertices of a regular simplex.
\item [(ii)]
For every centered Gaussian vector $(\xi_1,\dots,\xi_n)$ satisfying
\begin{equation}\label{eq:unit_var}
\E\xi_1^2=\dots =\E\xi_n^2=1,
\end{equation}
one has
\begin{equation}\label{eq:conj_for_exp}
\E\,\max\{\xi_1,\dots,\xi_n\}\leq\sqrt{\frac{n}{n-1}}\, \E\,\max\{\eta_1,\dots,\eta_n\},
\end{equation}
and the equality is attained iff $\E\, [\xi_i\xi_j] = -1/(n-1)$ for all $i\neq j$.
\end{itemize}
\end{proposition}
\begin{proof}
First of all note that
\begin{equation}\label{1603}
\sup_{x_1,\dots,x_{n}\in\S^{n-2}}V_1(\conv(x_1,\dots,x_{n}))=\sup_{y_1,\dots,y_{n}\in\S^{n-1}}V_1(\conv(y_1,\dots,y_{n}))
\end{equation}
because there is an $(n-1)$-dimensional affine subspace (and hence, an $(n-2)$-dimensional sphere of radius at most $1$) containing $y_1,\ldots,y_n$.
Therefore, we can restate (i) as follows:
\begin{equation}\label{2004_restate}
\sup_{y_1,\dots,y_{n}\in\S^{n-1}}V_1(\conv(y_1,\dots,y_{n}))=V_1(T_{n-1}),
\end{equation}
and the equality is attained iff $y_1,\ldots,y_n$ are vertices of a regular simplex centered at the origin. Let $\{e_1,\dots,e_n\}$ be a standard orthonormal basis in $\R^n$. As a realization of such a simplex we can take the convex hull of the points
$$
v_i := \sqrt\frac{n}{n-1}\left(e_i -\frac{e_1+\ldots+e_n}{n}\right), \quad i=1,\ldots,n.
$$
To see this, note that the $(n-1)$-dimensional regular  simplex
$$
S_{n-1}:=\conv(e_1,\dots,e_n)
$$
can be inscribed in an $(n-2)$-dimensional sphere of radius $\sqrt{(n-1)/n}$ centered at $(e_1+\ldots+e_n)/n$. It follows from~\eqref{1802} applied to $K=S_{n-1}$ that
\begin{equation}\label{1833}
V_1(T_{n-1}) = \sqrt{\frac{n}{n-1}}V_1(S_{n-1}) = \sqrt{2\pi}\sqrt{\frac{n}{n-1}}\E\,\max\{\eta_1,\dots,\eta_n\}.
\end{equation}

To any points $y_1,\dots,y_{n}\in\S^{n-1}$ we associate a centered Gaussian vector $(\xi_1,\dots,\xi_n)$  such that $\E\xi_1^2=\dots =\E\xi_n^2=1$ via
$$
\xi_1:=\langle \eta, y_1\rangle, \ldots, \xi_n:=\langle \eta, y_n\rangle.
$$
If we agree to identify two Gaussian vectors if they have the same distribution and two tuples $(y_1,\ldots,y_n)$ and $(y_1',\ldots,y_n')$ if $\langle y_i, y_j\rangle = \langle y_i',y_j'\rangle$ for all $i,j\in \{1,\ldots,n\}$, then this correspondence becomes one-to-one because $\Cov(\xi_i,\xi_j) = \langle y_i, y_j\rangle$.
It follows from~\eqref{1802} that
$$
\sqrt{2\pi}\, \E\,\max\{\xi_1,\dots,\xi_n\}=V_1(\conv(y_1,\dots,y_n)).
$$
The Gaussian vector corresponding to the points $v_1,\ldots,v_n$ satisfies
$$
\E\, [\xi_i \xi_j] = \langle v_i, v_j\rangle =-1/(n-1), \quad i\neq j.
$$
Taken together, the above considerations show the equivalence of (i) and (ii).
\end{proof}

\subsection{Asymptotic version of the conjecture}
We now show that~\eqref{2011} holds \emph{asymptotically}.

\begin{theorem}
For some absolute constant $C>0$ and all $n\in\N$,
\begin{align*}
V_1(T_n)
\leq
\sup_{x_1,\dots,x_{n+1}\in\S^{n-1}} V_1(\conv(x_1,\dots,x_{n+1}))
\leq
V_1(T_n)\left(1+C\frac{\log\log n}{\log n}\right).
  \end{align*}
\end{theorem}
\begin{proof}
The first inequality is trivial because we can take $x_1,\ldots, x_{n+1}$ to be the vertices of $T_n$.
Replacing $n$ by $n-1$ and using~\eqref{1603} we can restate that second inequality as follows: For all $n\geq 2$,
$$
\sup_{x_1,\dots,x_{n}\in\S^{n-1}} V_1(\conv(x_1,\dots,x_{n}))
\leq
V_1(T_{n-1})\left(1+C\frac{\log\log n}{\log n}\right)
$$
Fix $x_1,\dots,x_n\in\S^{n-1}$. For $k=1,\dots,n$ define Gaussian random variables $\xi_k:=\langle x_k,\eta\rangle$ and note that $\xi_k$ has zero mean and unit variance. It is known  (see, e.g., \cite[p.~138]{Chatterjee_book}) that
\begin{equation}\label{2104}
\E\,\max\{\xi_1,\dots,\xi_n\}\leq\sqrt{2\log n}.
\end{equation}
We provide a proof for the sake of completeness.
For $t>0$ one has
\begin{align*}
\exp\left( t\, \E\,\max\{\xi_1,\dots,\xi_n\} \right)    &
\leq \E\, \exp\left( t\max\{\xi_1,\dots,\xi_n\}\right)   \\
&=\E\,\max\{\eee^{t\xi_1},\dots,\eee^{t\xi_n}\}\leq\sum_{k=1}^n \E \, \eee^{t\xi_k}=n\eee^{t^2/2}.
\end{align*}
Letting $t=\sqrt{2\log n}$ yields \eqref{2104}.

On the other hand, it is well-known in the theory of extreme values, see~\cite[Theorem~1.5.3 on p.~14]{leadbetter_etal_book} and~\cite{pickands_moments},  that
 \begin{equation}\label{maximu}
\E\,\max\{\eta_1,\dots,\eta_n\}=\sqrt{2\log n}-O\left(\frac{\log\log n}{\sqrt{2\log n}}\right),\quad n\to\infty.
 \end{equation}
Using~\eqref{1802} and~\eqref{2104}, we obtain
$$
V_1(\conv(x_1,\dots,x_{n}))
= \sqrt{2 \pi}\,  \E \, \max\{\xi_1,\ldots,\xi_n\}
\leq \sqrt{4\pi\log n}.
$$
Combining this with~\eqref{1833} and~\eqref{maximu} gives
\begin{align*}
V_1(\conv(x_1,\dots,x_{n}))
&\leq V_1(T_{n-1})\cdot\sqrt{\frac{n-1}{n}}\left(1 - O\left(\frac{\log\log n}{\log n}\right)\right)^{-1}\\
&=V_1(T_{n-1})\cdot\left(1+O\left(\frac{\log\log n}{\log n}\right)\right),
\end{align*}
as $n\to\infty$. This proves the claim.
\end{proof}

\section{Regular simplex and regular crosspolytope}
In this section we compare the mean width of the regular simplex $T_{2n-1}$ to the mean width of the regular $n$-dimensional crosspolytope defined by
$$
C_n=\conv(\pm e_1,\dots,\pm e_n).
$$
Note that both $T_{2n-1}$ and $C_{n}$ (which can be considered as a degenerate simplex) have $2n$ vertices and can be inscribed in $\S^{2n-2}$. We will show that conjecture~\eqref{2011} is true in this special case, that is $V_1(C_{n}) \leq V_1(T_{2n-1})$.  Moreover, we will prove a lower bound which shows that  the mean width of $T_{2n-1}$ is remarkably close to the mean width of $C_n$.

\subsection{Mean width and extreme values}\label{subsec:extremes}
It follows from Sudakov's formula~\eqref{1802}, see also~\eqref{1833}, that
\begin{align}
&V_1(T_{n-1})
=
\sqrt{2\pi}\sqrt{\frac{n}{n-1}}\E\,\max\{\eta_1,\dots,\eta_n\}
=\sqrt{\frac{n}{n-1}} V_1(S_{n-1}),
\label{1833_simplex}\\
&V_1(C_n)
= \sqrt {2\pi}\, \E\,\max\{\pm \eta_1,\dots,\pm \eta_{n}\}
=\sqrt {2\pi}\, \E\,\max\{|\eta_1|,\dots,|\eta_{n}|\},
\label{1833_cross}
\end{align}
where we recall that $S_{n-1} = \conv(e_1,\ldots,e_n)$.
It is well-known in the theory of extreme values~\cite[Theorem~1.5.3 on p.~14]{leadbetter_etal_book} that
\begin{align}
&
\lim_{n\to\infty} \P\left[\max\{ \eta_1,\ldots, \eta_n\}\leq  u_n + \frac{x}{\sqrt{2\log n}} \right] =\eee^{-\eee^{-x}},\label{eq:max_gumbel1}\\
&\lim_{n\to\infty} \P\left[\max\{ |\eta_1|,\ldots, |\eta_n|\}\leq  u_{2n} + \frac{x}{\sqrt{2\log n}} \right] =\eee^{-\eee^{-x}}, \label{eq:max_gumbel2}
\end{align}
where $u_n$ is any sequence satisfying $\sqrt{2\pi} u_n \eee^{u_n^2/2} \sim n$, for example\footnote{$a_n \sim b_n$ means $a_n/b_n\to 1$ as $n\to \infty$.}
\begin{equation}\label{eq:def_u_n}
u_n = \sqrt{2\log n} - \frac{\frac 12 \log \log n +\log (2\sqrt{\pi})}{\sqrt{2\log n}}.
\end{equation}
Note that~\eqref{eq:max_gumbel2} (together with~\eqref{eq:max_gumbel1}) expresses the fact that the minimum and the maximum of $\eta_1,\ldots,\eta_n$ become asymptotically independent; see~\cite[Theorem~1.8.3 on p.~28]{leadbetter_etal_book}.
Taking the expectation (which is justified by~\cite{pickands_moments}) and noting that the expectation of the Gumbel distribution on the right-hand side of~\eqref{eq:max_gumbel1} and~\eqref{eq:max_gumbel2} is the Euler constant $\gamma$, we obtain the large $n$ asymptotics
\begin{align}
V_1(T_{n-1}) &= \sqrt{2\pi}\left( u_n + \frac{\gamma+o(1)}{\sqrt{2\log n}}\right), \quad n\to\infty,\label{eq:asympt_simplex} \\
V_1(C_n) &= \sqrt{2\pi}\left( u_{2n} + \frac{\gamma+o(1)}{\sqrt{2\log n}}\right), \quad n\to\infty. \label{eq:asympt_cross}
\end{align}
These formulas are known; see~\cite{affentranger_schneider}, \cite[p.~5]{finch_simplex}, \cite[p.~8]{finch_cross}.

\subsection{Comparing $V_1(T_{2n-1})$ and $V_1(C_n)$}
We are going to show that distance between $V_1(T_{2n-1})$ and $V_1(C_n)$ is in fact much smaller than
the bound $o(1/\sqrt{2\log n})$ implied by~\eqref{eq:asympt_simplex} and~\eqref{eq:asympt_cross}.
First we state the corresponding probabilistic result.
\begin{theorem}\label{1053}
If $\eta_1,\ldots, \eta_{2n}$ are independent standard Gaussian variables, then
\begin{equation*}
\E\,\max\{\eta_1,\dots,\eta_{2n}\}\leq
\E\,\max\{|\eta_1|,\dots,|\eta_n|\}\leq
\sqrt{\frac{2n}{2n-1}}\, \E\,\max\{\eta_1,\dots,\eta_{2n}\}.
\end{equation*}
\end{theorem}
The left hand-side inequality immediately follows from Slepian's lemma~\cite[Corollary~4.2.3 on p.~84]{leadbetter_etal_book} because the random vector $(\eta_1,\ldots,\eta_{2n})$ is uncorrelated, whereas the off-diagonal correlations of $(\eta_1,-\eta_1,\ldots, \eta_n,-\eta_n)$ are non-positive. The proof of the second estimate will be given in Section~\ref{sec:proof1}. 
Theorem~\ref{1053} together with~\eqref{1833_simplex} and~\eqref{1833_cross} implies the following
\begin{corollary}
For every $n\in\N$,
$$
\sqrt{\frac{2n-1}{2n}}V_1(T_{2n-1})\leq V_1(C_n)\leq V_1(T_{2n-1}).
$$
\end{corollary}

We now provide a bound which is asymptotically sharper. Its proof will be given in Section~\ref{sec:proof2}.

\begin{theorem}\label{1053_asympt}
Let $\eta_1,\eta_2,\dots,$ be independent standard Gaussian variables. Then, as $n\to\infty$, one has
$$
  \E\,\max\{|\eta_1|,\dots,|\eta_n|\} =
    \left( 1 + \frac{1+o(1)}{8n \log n}\right) \, \E\,\max\{\eta_1,\dots,\eta_{2n}\}.
$$
\end{theorem}
Combining Theorem~\ref{1053_asympt} with~\eqref{1833_simplex} and~\eqref{1833_cross} yields the following
\begin{corollary}
As $n\to\infty$,
$$
V_1(C_n) = V_1(T_{2n-1}) \left(1-\frac{1+o(1)}{4n}\right),
\quad
V_1(C_n) = V_1(S_{2n-1}) \left(1+\frac{1+o(1)}{8n\log n}\right).
$$
\end{corollary}
It is possible to obtain further asymptotic terms in~\eqref{eq:asympt_simplex} and~\eqref{eq:asympt_cross}, (see, e.g.,\ \cite[Eq.~(2.4.8) on p.~39]{leadbetter_etal_book}) but it seems that none of these expansions can correctly capture the very small difference between the expectations in Theorems~\ref{1053} and~\ref{1053_asympt}.

\section{Higher moments and limiting distribution of the random width}
\subsection{Sudakov's formula for higher moments}
Given a convex compact set $K\subset \R^n$ we denote by $W_K$ the length of the projection of $K$ onto a uniformly chosen direction, that is
\begin{equation}\label{eq:W_def_repeat}
W_K = \sup_{t\in K} \langle U,t \rangle - \inf_{t\in K} \langle U,t \rangle,
\end{equation}
where $U$ has uniform distribution on the sphere $\S^{n-1}$.  In this section we study the higher moments of the random variable $W_K$.

Recall that $\eta=(\eta_1,\ldots,\eta_n)$ denotes a random vector having standard normal distribution on $\R^n$. The \textit{isonormal Gaussian process} is defined by
$$
\Xi(t)=\langle \eta, t\rangle,\quad  t\in \R^n.
$$
It is characterized by
\begin{equation}\label{eq:isonormal_cov}
\E\, [\Xi(t)] = 0, \quad \E\, [\Xi(t) \Xi(s)] = \langle t, s\rangle, \quad t,s\in\R^n.
\end{equation}
For a compact set $K\subset \R^n$ define the range of $\Xi$ over $K$ to be
$$
\range_{t\in K} \Xi(t)
=
\sup_{t\in K} \Xi(t) - \inf_{t\in K} \Xi(t).
$$
The next theorem generalizes Sudakov's formula~\eqref{1802} to higher moments.

\begin{theorem}\label{theo:sudakov_general}
If the set $K\subset \R^n$ is convex and compact, then
\begin{equation}\label{eq:E_W_k_repeat}
\E [W^k_K] =
2^{-\frac k2} \frac{\Gamma\left(\frac n2\right)}{\Gamma\left(\frac{n+k}{2}\right)}
\E \left[ \left( \range_{t\in K} \Xi(t) \right)^k\right].
\end{equation}
If, moreover, the set $K$ is symmetric with respect to the origin,  then
\begin{equation}\label{eq:E_W_k}
\E [W_K^k]
=
2^{\frac k2} \frac{\Gamma\left(\frac n2\right)}{\Gamma\left(\frac{n+k}{2}\right)} \,
\E \left[\left(\sup_{t\in K}\Xi(t)\right)^k\right].
\end{equation}
\end{theorem}
Tsirelson~\cite{tsirelson2} generalized Sudakov's formula~\eqref{1802} to all intrinsic volumes. After the acceptance of this paper we have learned that Paouris and Pivovarov extended Tsirelson's formula to higher moments (see~\cite[Prop.~4.1]{PP13}) thereby proving a more general variant of Theorem~\ref{theo:sudakov_general}.
\begin{proof}
The standard Gaussian vector $\eta$ in $\R^n$ can be written as
$$
\eta \eqdistr RU,
$$
where $U$ and $R^2$ are such that
\begin{enumerate}
\item $U$ is a random vector with uniform distribution on the unit sphere in $\R^n$;
\item $R^2$ is a random variable having $\chi^2$-distribution with $n$ degrees of freedom;
\item $U$ and $R^2$ are independent.
\end{enumerate}
It follows that we have the distributional equality
\begin{equation}\label{eq:uniform_vs_gaussian_width}
\range_{t\in K} \Xi(t)
=
\sup_{t\in K} \langle \eta ,t \rangle - \inf_{t\in K} \langle \eta, t\rangle
\eqdistr
\sup_{t\in K} \langle RU ,t \rangle - \inf_{t\in K} \langle RU, t\rangle
=
RW_K.
\end{equation}
Taking $k$-th moments of both parts and noting that $R$ and $W_K$ are independent, we obtain that
$$
\E\left[\left(\range_{t\in K} \Xi(t)\right)^k\right] = \E [R^k]\, \E[W_K^k].
$$
The moments $\E [R^k]$ of the $\chi^2$-distribution are known. Inserting the value of the moment, we obtain~\eqref{eq:E_W_k_repeat} (which holds without the symmetry assumption on $K$). If the set $K$ is symmetric with respect to the origin, then $\range_{t\in K} \Xi(t) = 2\sup_{t\in K} \Xi(t)$ and we obtain~\eqref{eq:E_W_k}.
\end{proof}
\begin{remark}
In particular, taking $k=1$ in Theorem~\ref{theo:sudakov_general} and noting that the first intrinsic volume is related to the mean width $\E\, [W_K]$ by~\eqref{eq:mean_width_V1},
we recover from~\eqref{eq:E_W_k_repeat} Sudakov's~\cite{vS76} formula
\begin{equation}\label{eq:sudakov_repeat}
V_1(K) =  \sqrt{\frac{\pi}2}\, \E \left[\range_{t\in K} \Xi(t)\right] =  \sqrt {2\pi}\, \E \left[\sup_{t\in K} \Xi(t)\right].
\end{equation}
Note that the symmetry assumption on $K$ is not needed in the derivation of~\eqref{eq:sudakov_repeat} because in the last equality we used only that $\sup_{t\in K}\Xi(t)$ has the same distribution as $-\inf_{t\in K}\Xi(t)$.
\end{remark}

\subsection{Applications to regular polytopes}
Theorem~\ref{theo:sudakov_general} can be used to prove several conjectures on projections of regular polytopes which are due to Finch~\cite{finch_cross,finch_simplex,finch_width_distr}.
\begin{example}
Let $Q_n=[-\frac 12, +\frac 12]^n$ be the $n$-dimensional cube of unit volume. It is easy to see that $\range_{t\in Q_n} \Xi(t) = \sum_{i=1}^n |\eta_i|$. Therefore, by~\eqref{eq:E_W_k_repeat},
\begin{equation}\label{eq:moment_W_cube}
\E [W_{Q_n}^k]  = 2^{-\frac k2} \frac{\Gamma\left(\frac n2\right)}{\Gamma\left(\frac{n+k}{2}\right)} \E \left[\left(\sum_{i=1}^n |\eta_i|\right)^k\right].
\end{equation}
In particular, taking $k=1$ and noting that $\E |\eta_1|=\sqrt{\frac 2\pi}$ we obtain that the mean width is
$$
\E [W_{Q_n}] = \frac{\Gamma \left(\frac n2\right)}{\sqrt 2\, \Gamma\left(\frac{n+1}{2}\right)}\, n \, \E|\eta_1| =   \frac{n\, \Gamma \left(\frac n2\right)}{\sqrt {\pi} \, \Gamma\left(\frac{n+1}{2}\right)},
$$
or, equivalently, $V_1(Q_n) = n$, which is well known. The second moment of the projection length is given by
$$
\E [W^2_{Q_n}]
=
\frac{1}{n} \E \left[\left(|\eta_1|+\ldots+|\eta_n|\right)^2\right] = \E|\eta_1^2|  + (n-1) \E |\eta_1\eta_2|
=
1  + \frac{2}{\pi}(n-1),
$$
where we have used that $\E |\eta_1^2|=1$ and $\E |\eta_1\eta_2|=(\E |\eta_1|)^2 = \frac{2}{\pi}$.
This formula has been conjectured by Finch~\cite[p.~9]{finch_simplex} who established it for $n=2,3$ by purely geometric arguments~\cite{finch_width_distr}. Using~\eqref{eq:moment_W_cube} it is possible to compute more moments of $W_{Q_n}$, for example
\begin{align*}
\E [W_{Q_n}^3] &=  \frac{\Gamma\left(\frac n2\right)}{2\Gamma\left(\frac{n+3}{2}\right)} \pi^{-\frac 32} n \left(2n^2 + (3\pi-6)n + 4-\pi\right),\\
\E [W_{Q_n}^4] &= \frac{1}{(n+2)\pi^2} \left(4n^3 + (12\pi-24)n^2 + (44 - 20\pi + 3\pi^2) n +8\pi-24\right),
\end{align*}
where we have used that $\E |\eta_1| = \sqrt{\frac 2\pi}$,  $\E |\eta_1^2|=1$, $\E |\eta_1^3| = 2\sqrt{\frac 2\pi}$, $\E |\eta_1^4| = 3$.
\end{example}

\begin{example}
For the regular crosspolytope $C_n=\conv(\pm e_1,\ldots,\pm e_n)$ we have $\sup_{t\in C_n} \Xi(t) = \max\{|\eta_1|,\ldots,|\eta_n|\}$ and therefore Theorem~\ref{theo:sudakov_general} yields
$$
\E [W^k_{C_n}] = 2^{\frac k2} \frac{\Gamma\left(\frac n2\right)}{\Gamma\left(\frac{n+k}{2}\right)}
\E \left[ \left(\max_{1\leq i\leq n}  |\eta_i|\right)^k\right],
\quad k\in\N.
$$
For $k=2$, this formula was conjectured by Finch in~\cite[p.~3]{finch_cross}; see also~\cite{finch_simplex}.
\end{example}

\begin{example}
For the regular $(n-1)$-dimensional simplex $S_{n-1}=\conv(e_1,\ldots,e_n)\subset \R^n$, Theorem~\ref{theo:sudakov_general} yields
$$
\E [W^k_{S_{n-1}}] = 2^{-\frac k2} \frac{\Gamma\left(\frac n2\right)}{\Gamma\left(\frac{n+k}{2}\right)}
\E\left[\left(\max_{1\leq i\leq n} \eta_i -\min_{1\leq i\leq n} \eta_i\right)^k\right].
$$
Note that in this formula, $S_{n-1}$ is projected onto a random direction in $\R^n$, even though $S_{n-1}$ is $(n-1)$-dimensional.

It is more natural to state the corresponding formula for $T_{n-1}$ (which is a regular simplex with $n$ vertices inscribed in $\S^{n-2}\subset \R^{n-1}$) projected onto a random direction in $\R^{n-1}$. As a realization of $T_{n-1}$ we take the points
$$
v_i := \sqrt\frac{n}{n-1}\left(e_i -\frac{e_1+\ldots+e_n}{n}\right), \quad i=1,\ldots,n,
$$
in the hyperplane $L:= \{x_1+\ldots+x_n=0\}\subset \R^n$ (which can be identified with $\R^{n-1}$). By~\eqref{eq:isonormal_cov}, the isonormal process $\Xi$ on $L$ satisfies
$$
(\Xi(v_i))_{i=1,\ldots,n} \eqdistr  \sqrt\frac {n}{n-1} \left(\eta_i -\frac{\eta_1+\ldots+\eta_n}{n}\right)_{i=1,\ldots,n},
$$
so that for its range on $T_{n-1}$ we have
$$
\range_{t\in T_{n-1}} \Xi(t) \eqdistr \sqrt\frac {n}{n-1} \left(\max_{1\leq i\leq n} \eta_i -\min_{1\leq i\leq n} \eta_i\right).
$$
Therefore, for the projection length of  $T_{n-1}$ onto a uniformly chosen random direction in the hyperplane $L$ we obtain
$$
\E [W^k_{T_{n-1}}] =
2^{-\frac k2} \frac{\Gamma\left(\frac {n-1}2\right)}{\Gamma\left(\frac{n-1+k}{2}\right)} \left(\frac {n}{n-1}\right)^{k/2}
\E \left[\left(\max_{1\leq i\leq n} \eta_i -\min_{1\leq i\leq n} \eta_i\right)^k\right].
$$
For $k=2$, this formula was conjectured by Finch~\cite[p.~4]{finch_simplex} who verified it for small values of $n$.
\end{example}


\subsection{Limit distribution for the random width}
What is the asymptotic distribution of the projection length of a high-dimensional regular polytope onto a random line? The next two theorems  answer this question. The proofs are postponed to Section~\ref{sec:proofs3}.
\begin{theorem}\label{theo:shadow_cube_limit_distr}
The random width of the cube $Q_n =[-\frac 12, \frac 12]^n$ satisfies the following central limit theorem:
$$
W_{Q_n} -  \sqrt{\frac{2n}{\pi}} \todistr {\mathcal{N}} \left(0, \frac{\pi-3}{\pi}\right).
$$
\end{theorem}
After the acceptance of this paper we became aware of the reference~\cite{paourisetal} where the central limit theorem was established for the volume of the projection of the cube onto a random linear subspace of any fixed dimension.
\begin{theorem}\label{theo:shadow_simpl_cross_limit_distr}
For the random width of the simplex $S_{n-1} = \conv(e_1,\ldots,e_n)$ and the crosspolytope $C_n=\conv(\pm e_1,\ldots,\pm e_n)$ we have
\begin{align}
\sqrt{2n\log n} \left(W_{S_{n-1}} - \frac{2 u_{n}}{\sqrt n} \right)  &\todistr G_1+G_2\label{eq:limit_width_simpl},
\\
\sqrt{2n\log n} \left(W_{C_{n}} - \frac{2 u_{2n}}{\sqrt n} \right)  &\todistr 2G_1, \label{eq:limit_width_cross}
\end{align}
where $G_1,G_2$ are independent random variables with the Gumbel double exponential distribution $\P[G_1\leq x]= \P[G_2\leq x] = \eee^{-\eee^{-x}}$, $x\in\R$.
\end{theorem}
\begin{remark}
It is easy to check that the density of $G_1+G_2$ equals $2\eee^{-x} K_0(2\eee^{-x/2})$, $x\in\R$, where
$$
K_0(z)=\int_0^{\infty} \eee^{-z\cosh t} \dd t, \quad z>0,
$$
is the modified Bessel function of the second kind.
\end{remark}

\section{Proof of Theorem~\ref{1053}}\label{sec:proof1}
As already mentioned, the first estimate in Theorem~\ref{1053} is a consequence of the Slepian lemma. Therefore, we concentrate on proving the inequality
$$
\E\,\max\{|\eta_1|,\dots,|\eta_n|\}\leq
\sqrt{\frac{2n}{2n-1}}\, \E\,\max\{\eta_1,\dots,\eta_{2n}\}.
$$

For $n=1$ the inequality follows by direct calculations, thus we assume that $n\geq 2$.

The idea of the proof of goes back to the work of Chatterjee (see~\cite{sC05} or~\cite[p.~50]{AT09}). For $t\in[0,1]$ consider a centered Gaussian vector
$$
     \xi(t)=(\xi_1(t),\dots,\xi_{2n}(t))
$$
with correlations defined by
\begin{align*}
\E\,[\xi_i^2(t)]&=\frac{2n}{t+2n-1}, \quad i=1,\dots,2n,\\
\E\,[\xi_{2i-1}(t)\xi_{2i}(t)] &= -\frac{2nt}{t+2n-1}, \quad i=1,\dots,n,
\end{align*}
and $\E\,[\xi_i(t)\xi_j(t)]=0$ otherwise. We have
$$
\xi(0)\overset{d}{=}\sqrt{\frac{2n}{2n-1}}\, (\eta_1,\dots,\eta_{2n}), \quad \xi(1)\overset{d}{=}(\eta_1, -\eta_1, \eta_2, -\eta_2, \dots,\eta_n, -\eta_n).
$$
Hence it is sufficient to show that the function
$$
 \varphi(t):=\E\,\max\{\xi_1(t),\dots,\xi_{2n}(t)\}
$$
is non-increasing on $[0,1]$. Consider the function
$$
F_\beta(x_1,\dots,x_{2n}):=\frac1\beta\log\left(\sum_{i=1}^{2n}\eee^{\beta x_i}\right).
$$
It is immediate that
$$
\max\{x_1,\dots,x_{2n}\}\le F_\beta(x_1,\dots,x_{2n})\le \frac{\log 2n}{\beta}+\max\{x_1,\dots,x_{2n}\}.
$$
Therefore we only need to show that for any $\beta>0$ the function
$$
\varphi_\beta(t):=\E\,F_\beta(\xi(t))
$$
is non-increasing on $[0,1]$.

In what follows, $\bx$ stands for $(x_1,\dots,x_{2n})$. Set $\sigma_{ij}(t):=\E\,[\xi_i(t)\xi_j(t)]$ and let us denote by $f(t,\bx)$ the probability density function of $\xi(t)$. It is a well-known property of $f$ that
$$
\frac{\partial f}{\partial \sigma_{ii}}=\frac12\frac{\partial^2f}{\partial x^2_i},
\quad\quad
\frac{\partial f}{\partial \sigma_{ij}}=\frac{\partial^2f}{\partial x_i\partial x_j}, \quad i\ne j.
$$
Therefore,
$$
\frac{\partial\varphi_\beta}{\partial t}=\int_{\R^{2n}}F_\beta(\bx)\frac{\partial f}{\partial t}(\bx)\,\dd\bx=\int_{\R^{2n}}F_\beta(\bx)\sum_{1\le i\le j\le 2n}\frac{\partial f}{\partial \sigma_{ij}}(\bx)\frac{\partial \sigma_{ij}}{\partial t}\,\dd\bx.
$$
We have
\begin{align*}
\frac{\partial \sigma_{ii}}{\partial t}&=-\frac{2n}{(t+2n-1)^2}, \quad i=1,\dots,2n,\\
\frac{\partial \sigma_{2i-1,2i}}{\partial t}&=-\frac{2n(2n-1)}{(t+2n-1)^2}, \quad i=1,\dots,n,
\end{align*}
and $\partial\sigma_{ij}/\partial t=0$ otherwise. Thus we obtain
\begin{align*}
\frac{\partial\varphi_\beta}{\partial t}&=-\frac{n}{(t+2n-1)^2}\int_{\R^{2n}}F_\beta(\bx)\sum_{i=1}^n\left[\frac{\partial^2 f}{\partial x^2_{2i-1}}(\bx)+2(2n-1)\frac{\partial^2 f}{\partial x_{2i-1}\partial x_{2i}}(\bx)+\frac{\partial^2 f}{\partial x^2_{2i}}(\bx)\right]\,\dd\bx\\
&=-\frac{n}{(t+2n-1)^2}\sum_{i=1}^n\int_{\R^{2n}}f(\bx)\left[\frac{\partial^2 F_\beta}{\partial x^2_{2i-1}}(\bx)+2(2n-1)\frac{\partial^2 F_\beta}{\partial x_{2i-1} \partial x_{2i}}(\bx)+\frac{\partial^2 F_\beta}{\partial x^2_{2i}}(\bx)\right]\,\dd\bx\\
&=-\frac{n}{(t+2n-1)^2}\sum_{i=1}^n\E\,\left[\frac{\partial^2 F_\beta}{\partial x^2_{2i-1}}(\bx)+2(2n-1)\frac{\partial^2 F_\beta}{\partial x_{2i-1}\partial x_{2i}}(\bx)+\frac{\partial^2 F_\beta}{\partial x^2_{2i}}(\bx)\right].
\end{align*}
It is easy to check that
$$
\frac{\partial^2 F_\beta}{\partial x^2_{i}}(\bx)=\beta(p_i(\bx)-p_i^2(\bx)),
\quad\quad
\frac{\partial^2 F_\beta}{\partial x_{i}\partial x_{j}}(\bx)=-\beta p_i(\bx)p_j(\bx),
\quad i\ne j,
$$
where
$$
p_i(\bx):=\frac{\partial F_\beta}{\partial x_{i}}(\bx)=\frac{\eee^{\beta x_i}}{\sum_{k=1}^{2n}\eee^{\beta x_k}}.
$$
Thus,
\begin{align}
\lefteqn{-\frac{(t+2n-1)^2}{n\beta} \cdot  \frac{\partial\varphi_\beta}{\partial t}=} \label{2230}\\
&=
\sum_{i=1}^{2n}\E\, [p_{i}(\xi(t))] - \sum_{i=1}^{2n} \E\,[p_{i}^2(\xi(t))] - 2 (2n-1) \sum_{i=1}^{n} \E\, [p_{2i-1}(\xi(t))p_{2i}(\xi(t))]\notag \\
&= 1-\sum_{i=1}^{2n} \E\, [p_{i}^2(\xi(t))] - 2(2n-1) \sum_{i=1}^n \E\, [p_{2i-1}(\xi(t))p_{2i}(\xi(t))]. \notag
\end{align}
As we already mentioned, we assume that $n\geq 2$. For $i=1,\dots,n$ we have
$$
\E\, [p_{2i-1}(\xi(t))p_{j}(\xi(t))]=\left\{
                                        \begin{array}{ll}
                                          \E\, [p_{1}(\xi(t))p_{2}(\xi(t))], & \hbox{$j=2i$;} \\
                                          \E\, [p_{1}(\xi(t))p_{3}(\xi(t))], & \hbox{$j\ne 2i-1,2i$.}
                                        \end{array}
                                      \right.
$$
Therefore,
\begin{align*}
\lefteqn{-\frac{(t+2n-1)^2}{n\beta} \cdot  \frac{\partial\varphi_\beta}{\partial t}=} \\
&
1- \E\left[\sum_{i=1}^{2n}p_{i}(\xi(t))\right]^2-2n(2n-2)\E\, [p_{1}(\xi(t))p_{2}(\xi(t))]+2n(2n-2)\E\, [p_{1}(\xi(t))p_{3}(\xi(t))]\notag \\
&=-2n(2n-2)\E\, [p_{1}(\xi(t))p_{2}(\xi(t))]+2n(2n-2)\E\, [p_{1}(\xi(t))p_{3}(\xi(t))]. \notag
\end{align*}
Thus, to show that $ \frac{\partial\varphi_\beta}{\partial t}\le0$ and to complete the proof it is sufficient to prove that
$$
\E\, [p_{1}(\xi(t))p_{2}(\xi(t))]\le \E\, [p_{1}(\xi(t))p_{3}(\xi(t))],
$$
which is equivalent to
$$
\E\,\left[\frac{\eee^{\beta \xi_1(t)}(\eee^{\beta \xi_2(t)}-\eee^{\beta \xi_3(t)})}{(\sum_{i=1}^{2n}\eee^{\beta \xi_i(t)})^2}\right]\le0.
$$
Since the vectors $(\xi_1(t),\xi_2(t),\xi_3(t),\xi_4(t))$  and $(\xi_4(t),\xi_3(t),\xi_2(t),\xi_1(t))$ are equidistributed and independent from  $(\xi_5(t),\dots,\xi_{2n}(t))$, the last inequality is equivalent to
$$
\E\,\left[\frac{\eee^{\beta \xi_4(t)}(\eee^{\beta \xi_3(t)}-\eee^{\beta \xi_2(t)})}{(\sum_{i=1}^{2n}\eee^{\beta \xi_i(t)})^2}\right]\le0.
$$
Since the left hand sides of two last inequalities are equal, summing them up, we observe that it is enough to prove
\begin{equation}\label{1434}
\E\,\left[\frac{(\eee^{\beta \xi_1(t)}-\eee^{\beta \xi_4(t)})(\eee^{\beta \xi_2(t)}-\eee^{\beta \xi_3(t)})}{(\sum_{i=1}^{2n}\eee^{\beta \xi_i(t)})^2}\right]\le0.
\end{equation}
By the construction of vector $\xi(t)$, we have
$$
\frac{\E\,[\xi_{1}(t)\xi_{2}(t)]}{\sqrt{\E\,[\xi^2_{1}(t)]\E\,[\xi^2_{2}(t)]}}=-t
$$
and
$$
\sigma^2:=\E\,[\xi_i^2(t)]=\frac{2n}{t+2n-1}.
$$
Denote by $h(x_1,x_2,x_3,x_4)$ the probability density function of $(\xi_1(t),\xi_2(t),\xi_3(t),\xi_4(t))$:
\begin{equation}\label{1436}
h(x_1,x_2,x_3,x_4)=\frac{1}{4\pi^2\sigma^4(1-t^2)}\exp\left[-\frac{1}{2\sigma^2(1-t^2)}(x_1^2+x_2^2+x_3^2+x_4^2+2t(x_1x_2+x_3x_4))\right].
\end{equation}
Consider subsets $A,B\subset\R^4$ defined as
$$
A=\{(x_1,x_2,x_3,x_4)\in\R^4\,:\,(x_1-x_4)(x_2-x_3)>0\},
$$
$$
B=\{(x_1,x_2,x_3,x_4)\in\R^4\,:\,(x_1-x_4)(x_2-x_3)<0\}.
$$
We have
\begin{multline*}
\E\,\left[\frac{(\eee^{\beta \xi_1(t)}-\eee^{\beta \xi_4(t)})(\eee^{\beta \xi_2(t)}-\eee^{\beta \xi_3(t)})}{(\sum_{i=1}^{2n}\eee^{\beta \xi_i(t)})^2}\right]
\\=\E\,\left[\int_A\frac{(\eee^{\beta x_1}-\eee^{\beta x_4})(\eee^{\beta x_2}-\eee^{\beta x_3})}{(\eee^{\beta x_1}+\eee^{\beta x_2}+\eee^{\beta x_3}+\eee^{\beta x_4}+\sum_{i=5}^{2n}\eee^{\beta \xi_i(t)})^2}h(x_1,x_2,x_3,x_4)\dd x_1\dd x_2\dd x_3\dd x_4\right]
\\+\E\,\left[\int_B\frac{(\eee^{\beta x_1}-\eee^{\beta x_4})(\eee^{\beta x_2}-\eee^{\beta x_3})}{(\eee^{\beta x_1}+\eee^{\beta x_2}+\eee^{\beta x_3}+\eee^{\beta x_4}+\sum_{i=5}^{2n}\eee^{\beta \xi_i(t)})^2}h(x_1,x_2,x_3,x_4)\dd x_1\dd x_2\dd x_3\dd x_4\right].
\end{multline*}
Interchanging $x_2$ and $x_3$ in the second term, we get
\begin{multline*}
\E\,\left[\int_B\frac{(\eee^{\beta x_1}-\eee^{\beta x_4})(\eee^{\beta x_2}-\eee^{\beta x_3})}{(\eee^{\beta x_1}+\eee^{\beta x_2}+\eee^{\beta x_3}+\eee^{\beta x_4}+\sum_{i=5}^{2n}\eee^{\beta \xi_i(t)})^2}h(x_1,x_2,x_3,x_4)\dd x_1\dd x_2\dd x_3\dd x_4\right]
\\=-\E\,\left[\int_A\frac{(\eee^{\beta x_1}-\eee^{\beta x_4})(\eee^{\beta x_2}-\eee^{\beta x_3})}{(\eee^{\beta x_1}+\eee^{\beta x_2}+\eee^{\beta x_3}+\eee^{\beta x_4}+\sum_{i=5}^{2n}\eee^{\beta \xi_i(t)})^2}h(x_1,x_3,x_2,x_4)\dd x_1\dd x_2\dd x_3\dd x_4\right],
\end{multline*}
which yields
\begin{multline*}
\E\,\left[\frac{(\eee^{\beta \xi_1(t)}-\eee^{\beta \xi_4(t)})(\eee^{\beta \xi_2(t)}-\eee^{\beta \xi_3(t)})}{(\sum_{i=1}^{2n}\eee^{\beta \xi_i(t)})^2}\right]
\\=\E\,\left[\int_A\frac{(\eee^{\beta x_1}-\eee^{\beta x_4})(\eee^{\beta x_2}-\eee^{\beta x_3})(h(x_1,x_2,x_3,x_4)-h(x_1,x_3,x_2,x_4))}{(\eee^{\beta x_1}+\eee^{\beta x_2}+\eee^{\beta x_3}+\eee^{\beta x_4}+\sum_{i=5}^{2n}\eee^{\beta \xi_i(t)})^2}\dd x_1\dd x_2\dd x_3\dd x_4\right].
\end{multline*}
Since the exponent is increasing function and $\beta>0$, we have for $(x_1,x_2,x_3,x_4)\in A$
$$
(\eee^{\beta x_1}-\eee^{\beta x_4})(\eee^{\beta x_2}-\eee^{\beta x_3})\ge0.
$$
Thus, to get~\eqref{1434} it is enough to show that for $(x_1,x_2,x_3,x_4)\in A$
$$
h(x_1,x_2,x_3,x_4)-h(x_1,x_3,x_2,x_4)\le 0.
$$
Indeed, using~\eqref{1436} we obtain that for $(x_1,x_2,x_3,x_4)\in A$
\begin{align*}
\frac{h(x_1,x_3,x_2,x_4)}{h(x_1,x_2,x_3,x_4)}&=\exp\left[-\frac{t}{\sigma^2(1-t^2)}(x_1x_3+x_2x_4-x_1x_2-x_3x_4)\right]
\\&=\exp\left[\frac{t}{\sigma^2(1-t^2)}(x_1-x_4)(x_2-x_3)\right]\ge 1,
\end{align*}
which completes the proof.

\hfill$\Box$

\section{Proof of Theorem~\ref{1053_asympt}}\label{sec:proof2}
Recall that both $A_n:=\E\,\max_{1\leq i\leq n} |\eta_i|$  and $B_n := \E\,\max_{1\leq i\leq 2n} \eta_i$ are asymptotically equivalent to $\sqrt{2\log n}$. Therefore, in order to prove the theorem, we need to show that
\begin{equation}\label{eq:A_n_minus_B_n}
\lim_{n\to\infty}   4 n\sqrt{2\log n}\, \left(\E\,\max_{1\leq i\leq n} |\eta_i| - \E\,\max_{1\leq i\leq 2n} \eta_i\right) =1.
\end{equation}
Denote the tail function of the standard normal distribution by
$$
 \bar \Phi (t) = \int_t^{\infty} \eee^{-s^2/2}\, \frac{\dd s}{\sqrt{2\pi}}.
$$
It is well known~\cite[p.~9]{AT09} or~\cite[p.~137]{Chatterjee_book} that for $t>0$ one has
\begin{equation}\label{est}
\frac{1}{\sqrt{2\pi}} \left(\frac{1}{t} - \frac{1}{t^3}\right)  \, \eee^{-t^2/2}
\leq
\bar \Phi (t)
\leq
\frac{1}{\sqrt{2\pi}\, t} \, \eee^{-t^2/2}.
\end{equation}
The distribution functions of the maxima we are interested in are given by
\begin{align}
F_n(t)
&:=
\P \left[\max _{1\leq i\leq n}|\eta _i|\leq t \right]
=
\left(\int_{-t}^{t} \eee^{-s^2/2}\, \frac{\dd s}{\sqrt{2\pi}} \right)^n
=
(1-2\bar \Phi(t))^n, \quad t\geq 0, \label{eq:F_n_def}\\
G_n (t)
&:=
\P \left[\max _{1\leq i\leq 2n}\eta _i \leq t\right]
=
\left(\int_{-\infty}^{t} \eee^{-s^2/2}\, \frac{\dd s}{\sqrt{2\pi}} \right)^{2n}
=(1-\bar \Phi(t))^{2n}, \quad t\in\R. \label{eq:G_n_def}
\end{align}
It follows that
\begin{align*}
A_n
&=
\E\,\max_{1\leq i\leq n} |\eta_i|
 =
\int_0^{\infty} (1- F_n(t)) \, \dd t,\\
B_n
&=
\E\,\max_{1\leq i\leq 2n}\eta_i
=
\int_0^{\infty} (1- G_n(t)) \, \dd t - \int_0^{\infty} (\bar \Phi(t))^{2n} \, \dd t.
\end{align*}
To prove the second equality, note that for $M := \max _{1\leq i\leq 2n}\eta _i$ we have $M=M_+-M_-$ with $M_+=\max(M,0)$, $M_- = \max(-M,0)$, and
\begin{align*}
  \P \left[M_+ > t  \right]
  &= 1-G_n(t),\quad t > 0,\\
  \P \left[M_{-} > t  \right] &= \P \left[M <-t   \right]
  =(1-\bar \Phi(-t))^{2n} = (\bar \Phi(t))^{2n},\quad t > 0.
\end{align*}

In fact, the second integral in the formula for $B_n$ is negligible. Indeed, noting that $\bar \Phi(0)=1/2$ and using~\eqref{est} we obtain
\begin{align*}
\int_0^{\infty} (\bar \Phi(t))^{2n}\, \dd t
&\leq  (\bar \Phi(0))^{2n} +
\int_1^{\infty} (\bar \Phi(t))^{2n}\, \dd t
\leq
2^{-2n} +  \int_1^{\infty}
   \left(2\pi \right)^{-n} t^{-2n}
   \eee^{-t^2 n} \dd t\\
&\leq  2^{-2n} + \left(2\pi \eee \right)^{-n} \leq 2^{-n}.
\end{align*}

In view of the above considerations, in order to prove~\eqref{eq:A_n_minus_B_n} it suffices to show that
$$
\lim_{n\to\infty} 4 n\sqrt{2\log n} \int_{0}^{\infty} (G_n(t)-F_n(t)) \,\dd t = 1.
$$
After a change of variable $t := t_n + \frac a {t_n}$, $a\in\R$, where $t_n\sim \sqrt{2\log n}$ is the solution to the equation
\begin{equation}\label{eq:def_t_n}
\bar \Phi(t_n)= \frac 1{2n},
\end{equation}
our task reduces to proving that
\begin{equation}\label{eq:int1}
\lim_{n\to\infty}   \int_{-t_n^2}^{\infty} 4n \left(G_n\left(t_n + \frac a {t_n}\right)-F_n\left(t_n + \frac a {t_n}\right)\right) \,\dd a = 1.
\end{equation}

First we prove the pointwise convergence of the function under the integral sign. If $a\in\R$ is fixed and $n\to\infty$, then by~\eqref{est} and~\eqref{eq:def_t_n},
\begin{equation}\label{eq:r_n_asympt}
r_n := \bar \Phi\left(t_n + \frac a {t_n}\right)
\sim
\frac 1 {\sqrt{2\pi}\,t_n} \eee^{-\frac 12 \left(t_n + \frac a {t_n}\right)^2}
\sim
\frac 1 {\sqrt{2\pi}\,t_n} \eee^{-\frac 12 t_n^2} \eee^{-a}
\sim
\frac{\eee^{-a}}{2n}.
\end{equation}
Recalling the formulas for $F_n$ and $G_n$, see~\eqref{eq:F_n_def}, \eqref{eq:G_n_def}, we can write
\begin{align*}
F_n\left(t_n + \frac a {t_n}\right)
&=(1-2r_n)^{n}
=\eee^{n \log (1 - 2r_n)},\\
\quad
G_n\left(t_n + \frac a {t_n}\right)
&=(1-r_n)^{2n}
=\eee^{2n \log (1 - r_n)}.
\end{align*}
Using~\eqref{eq:r_n_asympt} and the Taylor series for the logarithm and the exponent, we obtain
\begin{align*}
F_n\left(t_n + \frac a {t_n}\right)
&=
\exp\left\{-n \left(2r_n + 2 r_n^2 + o\left(\frac 1{n^2}\right) \right)\right\}
=
\eee^{-2n r_n} \left(1- 2 n r_n^2+o\left(\frac 1{n}\right)\right),\\
G_n\left(t_n + \frac a {t_n}\right)
&=
\exp\left\{-2n \left(r_n + \frac {r_n^2} 2 + o\left(\frac 1{n^2}\right) \right)\right\}
=
\eee^{-2n r_n} \left(1 - n r_n^2+o\left(\frac 1{n}\right)\right).
\end{align*}
Subtracting both expansions and using~\eqref{eq:r_n_asympt} twice, we obtain
$$
4n \left(G_n\left(t_n + \frac a {t_n}\right)-F_n\left(t_n + \frac a {t_n}\right)\right)
 =
4n \eee^{-2n r_n} \left( n r_n^2+o\left(\frac 1{n}\right)\right)
\ton
\eee^{-\eee^{-a}} \eee^{-2a}.
$$
If we allow for a moment interchanging the limit and the integral, the limit in~\eqref{eq:int1} equals
$$
{\text{LHS of }}\eqref{eq:int1}
=
\int_{-\infty}^{+\infty} \eee^{-\eee^{-a}} \eee^{-2a} \,\dd a
=
\int_0^{\infty} \eee^{-y} y \,\dd y = 1,
$$
where we used the change of variable $y=\eee^{-a}$.

To complete the proof we need to justify the use of the Lebesgue dominated convergence theorem. It suffices to show that for some integrable function $g(a)$,
\begin{equation}\label{eq:dominated1}
0\leq n \left(G_n\left(t_n + \frac a {t_n}\right)- F_n\left(t_n + \frac a {t_n}\right)\right) \leq g(a), \quad a > - \frac 14t_n^2, \quad n\in\N,
\end{equation}
and
\begin{equation}\label{eq:dominated2}
\lim_{n\to\infty} n \int_{-t_n^2}^{-t_n^2/4} \left(G_n\left(t_n + \frac a {t_n}\right)-F_n\left(t_n + \frac a {t_n}\right)\right) \, \dd a = 0.
\end{equation}

The non-negativity of $G_n-F_n$ is a consequence of the inequality $(1-z)^2 \geq 1-2z$; see~\eqref{eq:F_n_def}, \eqref{eq:G_n_def}.
Now we prove the upper bound in~\eqref{eq:dominated1}. Using the inequality
$$y^n-x^n \leq n(y-x) y^{n-1}$$ for $0\leq x\leq y$, we obtain that
\begin{equation}\label{eq:g_n_f_n}
G_n(t) - F_n(t)
=
(1-2\bar \Phi(t) +\bar \Phi^2(t))^n - (1-2\bar \Phi(t))^n
\leq n
\bar \Phi^2 (t) (1-\bar \Phi(t))^{2n-2}.
\end{equation}
In the following, $C, C_1,\ldots>0$ denote absolute constants. Let first  $a>-\frac 14 t_n^2$ so that $t_n+\frac a{t_n}> \frac 34 t_n$.  By~\eqref{est} and~\eqref{eq:def_t_n},
\begin{equation}\label{eq:bar_Phi_upper}
\bar \Phi\left(t_n + \frac a {t_n}\right)
\leq
\frac {C_1} {t_n+\frac a{t_n}} \eee^{-\frac 12 \left(t_n + \frac a {t_n}\right)^2}
\leq
\frac {4C_1} {3t_n} \eee^{-\frac 12 t_n^2} \eee^{-a}
\leq
\frac{C_2}{n} \eee^{-a}.
\end{equation}
On the other hand, if $a\in [-\frac 14 t_n^2,0]$, then again using~\eqref{est} and~\eqref{eq:def_t_n} we obtain
\begin{equation}\label{eq:bar_Phi_lower}
\bar \Phi\left(t_n + \frac a {t_n}\right)
\geq
\frac {C_1'} {t_n+\frac a{t_n}} \eee^{-\frac 12 \left(t_n + \frac a {t_n}\right)^2}
\geq
\frac {C_1'} {t_n} \eee^{-\frac 12 t_n^2} \eee^{-a} \eee^{-\frac{a^2}{2t_n^2}}
\geq
\frac{C_2'}{n} \eee^{-\frac 78 a},
\end{equation}
where in the last estimate we used that $-\frac{a^2}{2t_n^2} \geq \frac 18 a$.

Note that because of $-a \leq \frac 14 t_n^2 \sim \frac 12\log n$, the right hand-side of~\eqref{eq:bar_Phi_upper} can be estimated above by $C n^{-1/4}$.
Using the inequality $1-x\leq \eee^{-\frac 12x}$ (which is valid for sufficiently small $x>0$) together with~\eqref{eq:bar_Phi_lower}, we  obtain that for $a\in [-\frac 14 t_n^2,0]$,
$$
\left(1-\bar \Phi\left(t_n+\frac a {t_n}
\right)\right)^{2n-2}
\leq
\eee^{- (n-1) \frac{C_2'}{n} \eee^{-\frac 78 a}}
\leq
\eee^{- C' \eee^{-\frac 78 a}}.
$$
It follows from this and~\eqref{eq:g_n_f_n}, \eqref{eq:bar_Phi_upper} that for all $a>-\frac 12 t_n^2$,
\begin{align*}
n (G_n-F_n)\left(t_n+\frac a {t_n}\right)
&\leq
n^2\bar \Phi^2 \left(t_n+\frac a {t_n}\right) \left(1-\bar \Phi\left(t_n+\frac a {t_n}
\right)\right)^{2n-2}\\
&\leq C'' \eee^{-2a} \eee^{- C'\eee^{-\frac 78 a}},
\end{align*}
where in the case $a>0$ we used  the trivial estimate $1-\bar \Phi(t)\leq 1$. The function on the right-hand side is integrable in $a$, thus completing the proof of~\eqref{eq:dominated1}.

We turn now to~\eqref{eq:dominated2}.  Using again~\eqref{eq:g_n_f_n}, the trivial estimate $\bar \Phi(t)\leq 1$, and the increasing property of $1-\bar \Phi(t)$, we obtain that
\begin{align*}
I_n := n \int_{-t_n^2}^{-t_n^2/4} (G_n-F_n)\left(t_n + \frac a {t_n}\right) \, \dd a
\leq
n^2 t_n^2 \left(1-\bar \Phi\left(\frac 34 t_n^2\right)\right)^{2n-2}.
\end{align*}
Recall now that $t_n^2\sim 2\log n$ and use~\eqref{eq:bar_Phi_lower} which implies that for some $\eps\in (0,1)$,
$$
\bar \Phi\left(\frac 34 t_n^2\right) \geq C n^{-\eps}, \text{ but } \lim_{n\to\infty}\bar \Phi\left(\frac 34 t_n^2\right) =0.
$$
Again using inequality $1-x\leq \eee^{-\frac 12x}$ (valid for small $x>0$), we obtain
$$
I_n \leq C n^2 (\log n) \eee^{- C n^{-\eps} (n-1)}  \ton 0,
$$
thus proving~\eqref{eq:dominated2}. \hfill $\Box$

\section{Proofs of Theorems~\ref{theo:shadow_cube_limit_distr} and~\ref{theo:shadow_simpl_cross_limit_distr}}\label{sec:proofs3}
Both proofs rely on the observation that a random vector $U$ distributed uniformly on $\S^{n-1}$ can be represented as
\begin{equation}\label{eq:U_rep}
U \eqdistr \left(\frac{\eta_1}{\sqrt{\eta_1^2 + \ldots +\eta_n^2}},\ldots, \frac{\eta_n}{\sqrt{\eta_1^2 + \ldots +\eta_n^2}} \right).
\end{equation}
\begin{proof}[Proof of Theorem~\ref{theo:shadow_cube_limit_distr}]
It follows from the definition of $W_{Q_n}$, see~\eqref{eq:W_def_repeat}, and from the central symmetry of the cube that
\begin{equation}\label{eq:W_Q_n_rep}
W_{Q_n} = 2 \sup_{t\in Q_n} \langle t, U \rangle \eqdistr \frac {|\eta_1|+\ldots+|\eta_n|}{\sqrt{\eta_1^2 + \ldots +\eta_n^2}}.
\end{equation}
Consider a random vector $(X_n,Y_n)$ with
\begin{equation}\label{eq:def_X_n_Y_n}
X_n := \frac{ |\eta_1| + \ldots + |\eta_n| - n \mu }{\sigma\sqrt n},
\quad
Y_n :=  \frac{ \eta_1^2 + \ldots + \eta_n^2 - n}{v\sqrt n},
\end{equation}
where
\begin{align}
\mu &:= \E |\eta_1| = \sqrt{2/\pi}, \label{eq:def_mu}\\
\sigma^2 &:= \Var |\eta_1| = \E [\eta_1^2] - (\E |\eta_1|)^2 = (\pi-2)/\pi, \label{eq:def_sigma}\\
v^2 &:= \Var (\eta_1^2) = \E [\eta_1^4] - (\E[\eta_1^2])^2 = 2. \label{eq:def_v}
\end{align}
Note that $\E X_n=\E Y_n = 0$ and  $\Var X_n=\Var Y_n=1$, while
\begin{equation}\label{eq:def_r}
r := \Cov (X_n,Y_n) = \frac{n \Cov (|\eta_1| ,\eta^2_1)}{\sigma v n} = \frac 1 {\sqrt{\pi-2}},
\end{equation}
where we used that $\E |\eta_1^3| = 2\sqrt{2/\pi}$. By the bivariate central limit theorem,
\begin{equation}\label{eq:CLT_bivariate}
(X_n, Y_n) \todistr (X,Y),
\end{equation}
where $(X,Y)$ is a bivariate Gaussian vector with standard margins and covariance $r$.
It follows from~\eqref{eq:def_X_n_Y_n} that
$$
W_{Q_n}\eqdistr
\frac {|\eta_1|+\ldots+|\eta_n|}{\sqrt{\eta_1^2 + \ldots +\eta_n^2}}
=
\frac{\mu n + \sigma \sqrt n X_n}{\sqrt{n + v\sqrt n Y_n}}
=
\frac{\mu n \left(1+ \frac{\sigma  X_n}{\mu\sqrt n}\right)}{\sqrt n\sqrt{1 + \frac{v Y_n}{\sqrt n}}}.
$$
Letting $n\to\infty$, expanding into a Taylor series around $0$ and noting that $X_n=O_P(1)$, $Y_n=O_P(1)$,  we get
$$
\frac {|\eta_1|+\ldots+|\eta_n|}{\sqrt{\eta_1^2 + \ldots +\eta_n^2}}
=
\mu \sqrt n \left( 1 + \frac 1 {\sqrt n}\left(\frac{\sigma X_n}{\mu} - \frac {v Y_n} 2\right) + O_P\left(\frac 1n\right)\right),
\quad n\to\infty.
$$
It follows that
$$
\frac {|\eta_1|+\ldots+|\eta_n|}{\sqrt{\eta_1^2 + \ldots +\eta_n^2}} -\mu \sqrt n
=
\sigma X_n - \frac 12 \mu v Y_n + O_P\left(\frac 1 {\sqrt n}\right),
\quad n\to\infty.
$$
Note that by the bivariate central limit theorem~\eqref{eq:CLT_bivariate}, the sequence $\sigma X_n - \frac 12 \mu v Y_n$ has limiting normal distribution with  mean zero and variance
$$
\Var \left[\sigma X_n - \frac 12 \mu v Y_n\right] = \sigma^2 +\frac 14 \mu^2 v^2 - \sigma \mu v r = \frac{\pi-3}{\pi},
$$
where we used~\eqref{eq:def_mu}, \eqref{eq:def_sigma}, \eqref{eq:def_v}, \eqref{eq:def_r}. Recalling~\eqref{eq:W_Q_n_rep} we obtain
$$
W_{Q_n} - \sqrt{\frac{2n}{\pi}} \eqdistr \frac {|\eta_1|+\ldots+|\eta_n|}{\sqrt{\eta_1^2 + \ldots +\eta_n^2}} -\mu \sqrt n
\todistr {\mathcal{N}} \left(0,\frac{\pi-3}{\pi}\right),
$$
which proves the claim.
\end{proof}
\begin{remark}
Self-normalized or studentized sums of the form
$$
R_n := \frac{\zeta_1+\ldots+\zeta_n}{\sqrt{\zeta_1^2+\ldots+\zeta_n^2}}
\quad \text{ or } \quad
T_n := 
\frac{\zeta_1+\ldots+\zeta_n}{\sqrt{(\zeta_1-\bar \zeta_n)^2+\ldots+(\zeta_n-\bar \zeta_n)^2}},
$$
where $\zeta_1,\zeta_2,\ldots$ are i.i.d.\ random variables and $\bar \zeta_n= \frac 1n (\zeta_1+\ldots+\zeta_n)$, have been extensively studied in the literature; see, e.g.,~\cite{pena_lai_shao_book}, with main emphasis on the central case $\E [\zeta_i] = 0$. The non-central case $\E [\zeta_i] \neq 0$ has been analyzed by~\citet{bentkus_etal} (who studied $T_n$) and by~\citet{omey_van_gulck} (who studied $1/R_n^2$ and related quantities). After some calculations, our central limit theorem for $W_{Q_n}$ could be deduced from~\cite[Theorem~3.1(v)]{omey_van_gulck}, but since these authors studied $1/R_n^2$ instead of $R_n$ it was easier to provide a direct proof.
\end{remark}

\begin{proof}[Proof of Theorem~\ref{theo:shadow_simpl_cross_limit_distr}]
We prove~\eqref{eq:limit_width_simpl}. Using representation~\eqref{eq:U_rep}, we obtain
\begin{equation}\label{eq:width_simplex_repre}
W_{S_{n-1}} \eqdistr \frac{\max_{1\leq i\leq n}\eta_i - \min_{1\leq i\leq n}\eta_i}{\sqrt{\eta_1^2+\ldots +\eta_n^2}}.
\end{equation}
It is known from extreme-value theory  that the range of the standard normal sample satisfies
\begin{equation}\label{eq:def_Z_n}
Z_n := u_n\left(\max_{1\leq i\leq n}\eta_i - \min_{1\leq i\leq n}\eta_i -  2u_{n}\right)\todistr G_1+G_2,
\end{equation}
where $u_n\sim\sqrt{2\log n}$ is as in~\eqref{eq:def_u_n}. In fact, this follows from the asymptotic independence~\cite[Theorem~1.8.3 on p.~28]{leadbetter_etal_book} of $\max_{1\leq i\leq n}\eta_i$ and $-\min_{1\leq i\leq n}\eta_i$ which both have limiting Gumbel distribution as in~\eqref{eq:max_gumbel1}.
Define $Y_n$ as in~\eqref{eq:def_X_n_Y_n} and observe that $Y_n$ has limiting standard normal distribution by the central limit theorem. We have
$$
\frac{\max_{1\leq i\leq n}\eta_i - \min_{1\leq i\leq n}\eta_i}{\sqrt{\eta_1^2+\ldots +\eta_n^2}}
=
\frac{2u_n + \frac{Z_n}{u_n}}{\sqrt{n + \sqrt{2n} Y_n}}
=
\frac{2u_n}{\sqrt n} \frac{1 + \frac{Z_n}{2u_n^2}}{\sqrt{1 + \sqrt{2/n} Y_n}}.
$$
Noting that both $Z_n$ and $Y_n$ are $O_P(1)$ and expanding into a Taylor series, we obtain
$$
\frac{\max_{1\leq i\leq n}\eta_i - \min_{1\leq i\leq n}\eta_i}{\sqrt{\eta_1^2+\ldots +\eta_n^2}}
=
\frac{2u_n}{\sqrt n} \left(1 + \frac{Z_n}{2u_n^2} + O_P\left(\frac 1{u_n^4}\right)\right),
$$
where we have used that $u_n\sim \sqrt{2\log n}$ and hence, the term with $Y_n$ is negligible. It follows from~\eqref{eq:def_Z_n} that
$$
u_n \sqrt n  \left(\frac{\max_{1\leq i\leq n}\eta_i - \min_{1\leq i\leq n}\eta_i}{\sqrt{\eta_1^2+\ldots +\eta_n^2}}  - \frac{2u_n}{\sqrt n}\right)
\todistr
G_1+G_2,
$$
which, in view of~\eqref{eq:width_simplex_repre}, implies~\eqref{eq:limit_width_simpl}.

The proof of~\eqref{eq:limit_width_cross} is analogous but instead of~\eqref{eq:W_Q_n_rep} it uses the representation
\begin{equation}\label{eq:width_cross_repre}
W_{C_{n}} \eqdistr \frac{2\max_{1\leq i\leq n}|\eta_i|}{\sqrt{\eta_1^2+\ldots +\eta_n^2}}.
\end{equation}
together with the limit relation
\begin{equation}\label{eq:def_Z_n_abs_value}
Z_n' := u_n \left(\max_{1\leq i\leq n}|\eta_i| - u_{2n}\right)\todistr G_1
\end{equation}
following from the asymptotic independence of the maximum and the minimum.
\end{proof}

{\bf Acknowledgments.} We are grateful to Manjunath Krishnapur and Tulasi Ram Reddy Annapareddy for noticing a gap in the first version of the proof of Theorem~\ref{1053}. We also wish to thank  Peter Pivovarov for bringing~\cite{PP13} to our attention.

\bibliographystyle{plainnat}
\bibliography{simpl_cross}

\end{document}